\documentclass[11pt,a4paper]{article}
\usepackage{mathrsfs}
\usepackage{bbding}
\usepackage{amsfonts}
\usepackage{epsf,epsfig,amsfonts,amsgen,amsmath,amstext,amsbsy,amsopn,amsthm,lineno}
\usepackage{color}
\newtheorem{theorem}{Theorem}

\newtheorem{lemma}{Lemma}
\newtheorem{false statement}{False Statement}
\newtheorem{corollary}{Corollary}

\theoremstyle{definition}

\newtheorem{claim}{Claim}

\newtheorem{conjecture}{Conjecture}

\newtheorem{case}{Case}

\newcounter{mathitem}
\newenvironment{mathitem}
  {\begin{list}{{$(\roman{mathitem})$}}{
   \setcounter{mathitem}{0}
   \usecounter{mathitem}
   \setlength{\topsep}{0pt plus 2pt minus 0pt}
   \setlength{\parskip}{0pt plus 2pt minus 0pt}
   \setlength{\partopsep}{0pt plus 2pt minus 0pt}
   \setlength{\parsep}{0pt plus 2pt minus 0pt}
   \setlength{\leftmargin}{20pt}
   \setlength{\itemsep}{0pt plus 2pt minus 0pt}}}
  {\end{list}}

\newcommand{\pa}{{\rm par}}

\begin{document}

\title{\bf\Large On path-quasar Ramsey numbers\thanks{Supported
by NSFC (No. 11271300) and the Doctorate Foundation of Northwestern
Polytechnical University (No. cx201202 and No. cx201326). E-mail
addresses: libinlong@mail.nwpu.edu.cn (B. Li),
ningbo\_math84@mail.nwpu.edu.cn (B. Ning).}}

\date{}

\author{Binlong Li$^{a,b}$ and Bo Ning$^a$\\[2mm]
\small $^a$ Department of Applied Mathematics,
\small Northwestern Polytechnical University,\\
\small Xi'an, Shaanxi 710072, P.R.~China\\
\small $^b$ Department of Mathematics, University of West Bohemia,\\
\small Univerzitn\'i 8, 306 14 Plze\v n, Czech Republic\\}
\maketitle

\begin{abstract}
Let $G_1$ and $G_2$ be two given graphs. The Ramsey number
$R(G_1,G_2)$ is the least integer $r$ such that for every graph $G$
on $r$ vertices, either $G$ contains a $G_1$ or $\overline{G}$
contains a $G_2$. Parsons gave a recursive formula to determine the
values of $R(P_n,K_{1,m})$, where $P_n$ is a path on $n$ vertices
and $K_{1,m}$ is a star on $m+1$ vertices. In this note, we first
give an explicit formula for the path-star Ramsey numbers. Secondly,
we study the Ramsey numbers $R(P_n,K_1\vee F_m)$, where $F_m$ is a
linear forest on $m$ vertices. We determine the exact values of
$R(P_n,K_1\vee F_m)$ for the cases $m\leq n$ and $m\geq 2n$, and for
the case that $F_m$ has no odd component. Moreover, we give a lower
bound and an upper bound for the case $n+1\leq m\leq 2n-1$ and $F_m$
has at least one odd component.

\medskip
\noindent {\bf Keywords:} Ramsey number; path; star; quasar
\smallskip

\noindent {\bf AMS Subject Classification:} 05C55, 05D10

\end{abstract}

\section{Introduction}

We use Bondy and Murty \cite{Bondy_Murty} for terminology and
notation not defined here, and consider finite simple graphs only.

Let $G$ be a graph. We denote by $\nu(G)$ the order of $G$, by
$\delta(G)$ the minimum degree of $G$, by $\omega(G)$ the number of
components of $G$, and by $o(G)$ the number of components of $G$
with an odd order.

Let $G_1$ and $G_2$ be two graphs. The \emph{Ramsey number}
$R(G_1,G_2)$, is defined as the least integer $r$ such that for
every graph $G$ on $r$ vertices, either $G$ contains a $G_1$ or
$\overline{G}$ contains a $G_2$, where $\overline{G}$ is the
complement of $G$. If $G_1$ and $G_2$ are both complete, then
$R(G_1,G_2)$ is the classical Ramsey number $r(\nu(G_1),\nu(G_2))$.
Otherwise, $R(G_1,G_2)$ is usually called the \emph{generalized
Ramsey number}. We refer the reader to Graham et. al.
\cite{Graham_Rothschild_Spencer} for an introduction to the area of
Ramsey theory.

We denote by $P_n$ the path on $n$ vertices. The graph $K_{1,m}$,
$m\geq 2$, is called a \emph{star}. The only vertex of degree $m$ is
called the \emph{center} of the star. In 1974, Parsons
\cite{Parsons} determined $R(P_n,K_{1,m})$ for all $n,m$. We list
Parsons' result as bellow.

\begin{theorem}[Parsons \cite{Parsons}]
$$R(P_n,K_{1,m})=\left\{\begin{array}{ll}
n,      & 2\leq m\leq\lceil n/2\rceil;\\
2m-1,   & \lceil n/2\rceil+1\leq m\leq n;\\
\max\{R(P_{n-1},K_{1,m}),R(P_n,K_{1,m-n+1})+n-1\},
        & n\geq 3 \mbox{ and } m\geq n+1.
\end{array}\right.$$
\end{theorem}

It is trivial that $R(P_2,K_{1,m})=m+1$. So the above recursive
formula can be used to determine all path-star Ramsey numbers.

In this note, we will first give an explicit formula for the Ramsey
numbers of paths versus stars. Let $t(n,m)$, $n,m\geq 2$, be the
values defined as
$$t(n,m)=\left\{\begin{array}{ll}
(n-1)\cdot\beta+1,  &\alpha\leq\gamma;\\
\lfloor(m-1)/\beta\rfloor+m,  &\alpha>\gamma,
\end{array}\right.$$
where
$$\alpha=\frac{m-1}{n-1}, \beta=\lceil\alpha\rceil \mbox{ and }
\gamma=\frac{\beta^2}{\beta+1}.$$

\begin{theorem}
$R(P_n,K_{1,m})=t(n,m)$ for all $n,m\geq 2$.
\end{theorem}

The interested reader can compare our formula with Parsons' theorem.
We will give an independent and short proof of Theorem 2 in Section
3.

A \emph{linear forest} is a forest each component of which is a
path. We call the graph obtained by joining a vertex to every vertex
of a nontrivial linear forest a \emph{quasar}. Thus a star is a
quasar, and we call a quasar a \emph{proper} one if it is not a
star.

It may be interesting to consider the Ramsey numbers of paths versus
proper quasars. Some results of this area were obtained. Salman and
Broersma \cite{Salman_Broersma_1,Salman_Broersma_2} studied the
Ramsey numbers of $P_n$ versus $K_1\vee mK_2$ (this graph is called
a \emph{fan} in \cite{Salman_Broersma_1}), and of $P_n$ versus
$K_1\vee P_m$ (this graph is called a \emph{kipas} in
\cite{Salman_Broersma_2}). Both cases have not been completely
solved in \cite{Salman_Broersma_1,Salman_Broersma_2}. Note that fans
and kipases are special cases of quasars. In the following, we will
consider the Ramsey numbers of paths versus proper quasars. As an
application of our results, we will give a complete solution to the
problem of determining the Ramsey numbers of paths versus fans.

We first determine the exact values of $R(P_n,K_1\vee F)$ when
$m\leq n$ or $m\geq 2n$, where $F$ is a non-empty linear forest on
$m$ vertices.

\begin{theorem}
Let $F$ be a non-empty linear forest on $m$ vertices. Then
$$R(P_n,K_1\vee F)=\left\{\begin{array}{ll}
2n-1,       & 2\leq m\leq n;\\
t(n,m),     & n\geq 2 \mbox{ and }m\geq 2n.
\end{array}\right.$$
\end{theorem}

So we have an open problem for the case $n+1\leq m\leq 2n-1$. For
this case we have the following upper and lower bounds. By $\pa(m)$
we denote the parity of $m$.

\begin{theorem}
If $n\geq 2$ and $n+1\leq m\leq 2n-1$, and $F$ is a non-empty linear
forest on $m$ vertices, then
\begin{mathitem}
\item[(1)] $R(P_n,K_1\vee F)\leq m+n-2+\pa(m)$; and
\item[(2)] $R(P_n,K_1\vee F)\geq\max\left\{2n-1,\lceil 3m/2\rceil-1,
m+n-o(F)-2\right\}$.
\end{mathitem}
\end{theorem}

If $F$ contains no odd component, then the upper bound and the lower
bound in Theorem 4 are equal. Thus we conclude the following.

\begin{corollary}
If $n\geq 2$ and $n+1\leq m\leq 2n-1$, and $F$ is a linear forest on
$m$ vertices such that each component of $F$ has an even order, then
$$R(P_n,K_1\vee F)=m+n-2.$$
\end{corollary}

Note that Theorem 3 and Corollary 1 give all the path-quasar Ramsey
numbers $R(P_n,K_1\vee F)$ when $o(F)=0$, including all the Ramsey
numbers of paths versus fans.

\begin{corollary}
$$R(P_n,K_1\vee mK_2)=\left\{\begin{array}{ll}
2n-1,       & 1\leq m\leq\lfloor n/2\rfloor;\\
m+n-2,      & \lfloor n/2\rfloor+1\leq m\leq n-1; \\
t(n,m),     & n\geq 2 \mbox{ and }m\geq n.
\end{array}\right.$$
\end{corollary}

We propose the following conjecture to complete this section.

\begin{conjecture}
If $n\geq 2$ and $n+1\leq m\leq 2n-1$, and $F$ is a non-empty linear
forest on $m$ vertices, then
$$R(P_n,K_1\vee F)=\max\left\{2n-1,\left\lceil\frac{3m}{2}\right\rceil-1,
m+n-o(F)-2\right\}.$$
\end{conjecture}

\section{Preliminaries}

The following useful result is deduced from Dirac \cite{Dirac}. We
present it here without a proof.

\begin{theorem}
  Every connected graph $G$ contains a path of order at
  least $\min\{\nu(G),2\delta(G)+1\}$.
\end{theorem}

We follow the notation in \cite{Li_Ning}. For integers $s,t$, the
\emph{interval} $[s,t]$ is the set of integers $i$ with $s\leq i\leq
t$. Note that if $s>t$, then $[s,t]=\emptyset$. Let $X$ be a subset
of $\mathbb{N}$. We set $\mathcal {L}(X)=\{\sum_{i=1}^kx_i: x_i\in
X, k\in \mathbb{N}\}$, and suppose $0\in\mathcal{L}(X)$ for any set
$X$. Note that if $1\in X$, then $\mathcal{L}(X)=\mathbb{N}$. For an
interval $[s,t]$, we use $\mathcal{L}[s,t]$ instead of
$\mathcal{L}([s,t])$.

The following lemma was proved by the authors in \cite{Li_Ning}. We
include the proof here for the completeness of our discussions.

\begin{lemma}
  $t(n,m)=\min\{t: t\notin \mathcal {L}[t-m+1,n-1]\}$.
\end{lemma}

\begin{proof}
Set $T=\{t: t\in\mathcal{L}[t-m+1,n-1]\}$. Note that if $t\in T$,
then $t-1\in T$. So it is sufficient to prove that
$t(n,m)=\max(T)+1$.

Note that
\begin{align*}
                & t\in T \Leftrightarrow t\in\mathcal{L}[t-m+1,n-1]\\
\Leftrightarrow & t\in[k(t-m+1),k(n-1)], \mbox{ for some integer } k\\
\Leftrightarrow & t\leq\frac{k}{k-1}(m-1) \mbox{ and } t\leq k(n-1),
                    \mbox{ for some integer } k\\
\Leftrightarrow & t\leq k(n-1) \mbox{ for some integer } k<\alpha+1,
                    \mbox{ or}\\
                & t\leq\left\lfloor\frac{m-1}{k-1}\right\rfloor+m-1,
                    \mbox{ for some integer } k\geq\alpha+1.
\end{align*}
This implies that
$$T=\left\{t: t\leq k(n-1), k\leq\beta\right\}\cup\left\{t:
t\leq\left\lfloor\frac{m-1}{k-1}\right\rfloor+m-1,
k\geq\beta+1\right\}.$$ Thus
\begin{align*}
\max(T) & =\max\left\{(n-1)\beta,
            \left\lfloor\frac{m-1}{\beta}\right\rfloor+m-1\right\}\\
        & =\left\{\begin{array}{ll}
            (n-1)\cdot\beta,                &\alpha\leq\gamma;\\
            \lfloor(m-1)/\beta\rfloor+m-1,  &\alpha>\gamma.
            \end{array}\right.
\end{align*}
We conclude that $t(n,m)=\max(T)+1$.
\end{proof}

We use $C_m$ to denote the cycle on $n$ vertices, and $W_m$ to
denote the wheel on $m+1$ vertices, i.e., the graph obtained by
joining a vertex to every vertex of a $C_m$. We will use the
following formulas for path-cycle Ramsey numbers and for path-wheel
Ramsey numbers.

\begin{theorem}[Faudree et al. \cite{Faudree_Lawrence_Parsons_Schelp}]
If $n\geq 2$ and $m\geq 3$, then
$$R(P_n,C_m)=\left\{\begin{array}{ll}
2n-1,           & \mbox{for } n\geq m \mbox{ and } m \mbox{ is odd};\\
n+m/2-1,        & \mbox{for } n\geq m \mbox{ and } m \mbox{ is even};\\
\max\{m+\lfloor n/2\rfloor-1,2n-1\},    & \mbox{for } m>n
                    \mbox{ and } m \mbox{ is odd};\\
m+\lfloor n/2\rfloor-1,     & \mbox{for } m>n \mbox{ and } m \mbox{
                    is even}.
\end{array}\right.$$
\end{theorem}

\begin{theorem}
\
\begin{mathitem}
\item[(1)] (Chen et al. \cite{Chen_Zhang_Zhang}) If $3\leq m\leq n+1$, then
$$R(P_n,W_m)=\left\{\begin{array}{ll}
3n-2,   &m \mbox{ is odd};\\
2n-1,   &m \mbox{ is even}.
\end{array}\right.$$
\item[(2)] (Zhang \cite{Zhang}) If $n+2\leq m\leq 2n$, then
$$R(P_n,W_m)=\left\{\begin{array}{ll}
3n-2,   &m \mbox{ is odd};\\
m+n-2,  &m \mbox{ is even}.
\end{array}\right.$$
\item[(3)] (Li and Ning \cite{Li_Ning}) If $n\geq 2$ and $m\geq 2n+1$, then
$$R(P_n,W_m)=t(n,m).$$
\end{mathitem}
\end{theorem}

\section{Proofs of the theorems}

\textbf{Proof of Theorem 2.} Let $r=t(n,m)$. By Lemma 1,
$t(n,m)=\min\{t: t\notin \mathcal {L}[t-m+1,n-1]\}$. Thus $r-1\in
\mathcal {L}[r-m,n-1]$. Let $r-1=\sum_{i=1}^k r_i$, where
$r_i\in[r-m,n-1]$, $1\leq i\leq k$. Let $G$ be a graph with $k$
components $H_1,\ldots,H_k$ such that $H_i$ is a clique on $r_i$
vertices. Note that $G$ contains no $P_n$ since every component of
$G$ has less than $n$ vertices; and $\overline{G}$ contains no
$K_{1,m}$ since every vertex of $G$ has less than $m$ nonadjacent
vertices. This implies that $R(P_n,K_{1,m})\geq\nu(G)+1=r$.

Now we will prove that $R(P_n,K_{1,m})\leq r$. Assume not. Let $G$
be a graph on $r$ vertices such that $G$ contains no $P_n$ and
$\overline{G}$ contains no $K_{1,m}$.

\begin{claim}
$m+\lfloor n/2\rfloor\leq r\leq m+n-1$, i.e., $1\leq m+n-r\leq\lceil
n/2\rceil$.
\end{claim}

\begin{proof}
Let $r'=m+n-1$. Since $r'-m+1=n$, $[r'-m+1,n-1]=\emptyset$, and
$r'\notin \mathcal {L}(\emptyset)=\{0\}$, we have $r\leq r'=m+n-1$
and hence $m+n-r\geq 1$.

Now we prove that $m+n-r\leq(n+1)/2$. By Lemma 1, $r\notin \mathcal
{L}[r-m+1,n-1]$. Thus $r\notin[k(r-m+1),k(n-1)]$, for every $k$.
That is, $r\in[k(n-1)+1,(k+1)(r-m+1)-1]$, for some $k$. This implies
that
$$r\geq k(n-1)+1 \mbox{ and } r\geq\frac{k+1}{k}m-1,$$
for some $k\geq 1$.

If $m\leq(k^2n-k^2+2k)/(k+1)$, then
\begin{align*}
    & m+n-r\leq \frac{k^2n-k^2+2k}{k+1}+n-(k(n-1)+1)\\
=   & \frac{n+2k-1}{k+1}\leq\frac{n+1}{2}.
\end{align*}

If $m>(k^2n-k^2+2k)/(k+1)$, then
\begin{align*}
    & m+n-r\leq m+n-\left(\frac{k+1}{k}m-1\right)\\
=   & n-\frac{m}{k}+1<n-\frac{k^2n-k^2+2k}{k(k+1)}+1\\
=   & \frac{n+2k-1}{k+1}\leq\frac{n+1}{2}.
\end{align*}

Thus we have $m+n-r\leq\lfloor(n+1)/2\rfloor=\lceil n/2\rceil$.
\end{proof}

\begin{case}
  Every component of $G$ has order less than $n$.
\end{case}

Let $H_i$, $1\leq i\leq k=\omega(G)$, be the components of $G$.
Since $r\notin \mathcal {L}[r-m+1,n-1]$, there is a component, say
$H_1$, with order at most $r-m$. Thus $\sum_{i=2}^k\nu(H_i)\geq m$.
Let $v$ be a vertex in $H_1$. Since $v$ is nonadjacent to every
vertex in $G-H_1$, $\overline{G}$ contains a $K_{1,m}$ with the
center $v$, a contradiction.

\begin{case}
  There is a component of $G$ with order at least $n$.
\end{case}

Let $H$ be a component of $G$ with $\nu(H)\geq n$. If every vertex
of $H$ has degree at least $\lfloor n/2\rfloor$, then by Theorem 5,
$H$ contains a $P_n$, a contradiction. Thus there is a vertex $v$ in
$H$ with $d(v)\leq\lfloor n/2\rfloor-1$. Let $G'=G-v-N(v)$. Then by
Claim 1, $$\nu(G')=\nu(G)-1-d(v)\geq
r-\left\lfloor\frac{n}{2}\right\rfloor\geq m.$$ Since $v$ is
nonadjacent to every vertex in $G'$, $\overline{G}$ contains a
$K_{1,m}$ with the center $v$, a contradiction.

The proof is complete. \hfill$\Box$

\noindent\textbf{Proof of Theorem 3.} If $m=2$, then $K_1\vee F$ is
a triangle (recall that $F$ is non-empty). From Theorem 6, we get
that $R(P_n,C_3)=2n-1$.

If $3\leq m\leq n$, then $K_1\vee F$ is a supergraph of $C_3$ and a
subgraph of $W_{m+\pa(m)}$, we have
$$R(P_n,C_3)\leq R(P_n,K_1\vee F)\leq R(P_n,W_{m+\pa(m)}).$$
By Theorems 6 and 7, $R(P_n,C_3)=R(P_n,W_{m+\pa(m)})=2n-1$. We
conclude that $R(P_n,K_1\vee F)=2n-1$.

Now we deal with the case $m\geq 2n$. Note that $K_1\vee F$ is a
supergraph of $K_{1,m}$ and a subgraph of $W_m$. We have
$$R(P_n,K_{1,m})\leq R(P_n,K_1\vee F)\leq R(P_n,W_m).$$
By Theorems 2 and 7, $R(P_n,K_{1,m})=R(P_n,W_m)=t(n,m)$ (we remark
that if $m=2n$, then $m+n-2=t(n,m)$). We conclude that
$R(P_n,K_1\vee F)=t(n,m)$.

The proof is complete. \hfill$\Box$

\noindent\textbf{Proof of Theorem 4.} Since $K_1\vee F$ is a
subgraph of $W_{m+\pa(m)}$, by Theorem 7, we have
$$R(P_n,K_1\vee F)\leq m+n-2+\pa(m).$$

Now we construct three graphs. Let
$$\begin{array}{l}
G_1=2K_{n-1},\\
G_2=K_{\lfloor m/2\rfloor}\cup 2K_{\lceil m/2\rceil-1} \mbox{ and}\\
G_3=K_{n-1}\cup 2K_{(m-o(F))/2-1}.
\end{array}$$ One can check that all the three graphs contain no
$P_n$ and their complements contain no $K_1\vee F$. This implies
that $R(P_n,K_1\vee F)\geq\max\{\nu(G_i)+1: i=1,2,3\}$. Since
$\nu(G_1)=2n-2$, $\nu(G_2)=\lceil 3m/2\rceil-2$ and
$\nu(G_3)=m+n-o(F)-3$, we get that
$$R(P_n,K_1\vee F)\geq\max\left\{2n-1,\left\lceil\frac{3m}{2}\right\rceil-1,
m+n-o(F)-2\right\}.$$

The proof is complete. \hfill$\Box$


\begin{thebibliography}{10}

\bibitem{Bondy_Murty}
J.A.~Bondy and U.S.R.~Murty, Graph Theory with Applications,
Macmillan, London and Elsevier, New York, 1976.

\bibitem{Chen_Zhang_Zhang}
Y. Chen, Y. Zhang and K. Zhang, The Ramsey numbers of paths versus
wheels, \emph{Discrete Math.} {\bf 290} (1) (2005) 85--87.

\bibitem{Dirac}
G.A. Dirac, Some theorems on abstract graphs, \emph{Proc. London.
Math. Soc.} {\bf 2} (1952) 69--81.

\bibitem{Faudree_Lawrence_Parsons_Schelp}
R.J. Faudree, S.L. Lawrence, T.D. Parsons and R.H. Schelp,
Path-cycle Ramsey numbers, \emph{Discrete Math.} {\bf 10} (2) (1974)
269--277.

\bibitem{Graham_Rothschild_Spencer}
R.L. Graham, B.L. Rothschild and J.H. Spencer, Ramsey Theory, Second
Edition, John Wiley \& Sons Inc., New York, 1990.

\bibitem{Li_Ning}
B. Li and B. Ning, The Ramsey numbers of paths versus wheels: a
complete solution, preprint available at
http://arxiv.org/abs/1312.2081.

\bibitem{Parsons}
T.D. Parsons, Path-star Ramsey numbers, \emph{J. Combin. Theory,
Ser. B} {\bf 17} (1) (1974) 51--58.

\bibitem{Salman_Broersma_1}
A.N.M. Salman and H.J. Broersma, Path-fan Ramsey numbers,
\emph{Discrete Applied Math.} {\bf 154} (9) (2006) 1429--1436.

\bibitem{Salman_Broersma_2}
A.N.M. Salman and H.J. Broersma, Path-kipas Ramsey numbers,
\emph{Discrete Applied Math.} {\bf 155} (14) (2007) 1878--1884.

\bibitem{Zhang}
Y. Zhang, On Ramsey numbers of short paths versus large wheels,
\emph{Ars Combin.} {\bf 89} (2008) 11--20.

\end{thebibliography}
\end{document}